\documentclass[12pt]{amsart}
\usepackage{amssymb}

\newtheorem{theorem}{Theorem}[section]
\newtheorem{lemma}[theorem]{Lemma}
\newtheorem{proposition}[theorem]{Proposition}
\newtheorem{corollary}[theorem]{Corollary}

\theoremstyle{definition}

\numberwithin{equation}{section}

\begin{document}

\baselineskip=15.5pt

\title[On the K\"ahler structures over Quot schemes, II]{On the K\"ahler structures
over Quot schemes, II}

\author[I. Biswas]{Indranil Biswas}

\address{School of Mathematics, Tata Institute of Fundamental
Research, Homi Bhabha Road, Bombay 400005, India}

\email{indranil@math.tifr.res.in}

\author[H. Seshadri]{Harish Seshadri}

\address{Indian Institute of Science, Department of Mathematics,
Bangalore 560003, India}

\email{harish@math.iisc.ernet.in}



\date{}

\begin{abstract}
Let $X$ be a compact connected Riemann surface of genus $g$, with
$g\, \geq\, 2$, and let ${\mathcal O}_X$ denote the sheaf of holomorphic functions on $X$. Fix
positive integers $r$ and $d$ and let ${\mathcal Q}(r,d)$ be the Quot scheme
parametrizing all torsion coherent quotients of ${\mathcal O}^{\oplus r}_X$ of degree $d$. We 
prove that ${\mathcal Q}(r,d)$ does not admit a K\"ahler metric whose holomorphic 
bisectional curvatures are all nonnegative.
\end{abstract}

\maketitle

\section{Introduction}

Let $X$ be a compact connected Riemann surface of genus $g$, with $g\, \geq\, 2$. For
any positive integer $d$, let $S^d(X)$ denote the $d$-fold symmetric product of $X$.
The main theorem of \cite{BR} says the following (see \cite[Theorem 1.1]{BR}):
If $d\, \leq\, 2(g-1)$, then $S^d(X)$ does not admit any K\"ahler
metric for which all the holomorphic bisectional curvatures are nonnegative.

A natural generalization of the symmetric product $S^d(X)$ is the Quot scheme
${\mathcal Q}(r,d)$ that parametrizes all torsion coherent
quotients of ${\mathcal O}^{\oplus r}_X$ of degree $d$. So, ${\mathcal Q}(1,d)\,
=\, S^d(X)$. These Quot schemes arise
naturally in the study of vector bundles on curves (see \cite{BGL}). They also
appear in mathematical physics as moduli spaces of vortices (cf. \cite{Ba},
\cite{BiR}, \cite{BR}).

Our aim here is to prove the following (see Theorem \ref{thm1}):

\textit{The complex manifold ${\mathcal Q}(r,d)$ does not admit any K\"ahler
metric such that all the holomorphic bisectional curvatures are nonnegative.}

In \cite{BS}, this was proved under the assumption that $d\, \leq \, 2(g-1)$.

\section{The Albanese map for $\mathcal Q$}\label{se2}

Let $X$ be a compact connected Riemann surface of genus $g$, with
$g\, \geq\, 2$. The sheaf of germs of holomorphic functions on $X$ will be denoted
by ${\mathcal O}_X$. Fix positive integers $r$ and $d$. Let
$$
{\mathcal Q}\, :=\, {\mathcal Q}(r,d)
$$
be the Quot scheme that parametrizes all torsion coherent quotients of ${\mathcal
O}^{\oplus r}_X$ of degree $d$. In other words, ${\mathcal Q}$ parametrizes all
coherent subsheaves of ${\mathcal O}^{\oplus r}_X$ of rank $r$ and degree $-d$.

The group of permutations of $\{1\, ,\cdots\, ,d\}$ will be denoted by
$P_d$. Let $S^d(X)\,:=\, X^d/P_d$ be the symmetric product of $X$. The
Quot scheme ${\mathcal Q}(1,d)$ is identified with $S^d(X)$ by sending
any quotient $Q'\, \in\, {\mathcal Q}(1,d)$ of ${\mathcal O}_X$ to the
scheme-theoretic support of $Q'$. Take any quotient $Q\, \in\,
{\mathcal Q}\, =\, {\mathcal Q}(r,d)$. Consider the corresponding short
exact sequence
$$
0\, \longrightarrow\, {\mathcal K}_Q \, \longrightarrow\,
{\mathcal O}^{\oplus r}_X\, \longrightarrow\, Q\, \longrightarrow\, 0
$$
on $X$. Let
$$
0\, \longrightarrow\, \bigwedge\nolimits^r {\mathcal K}_Q \, \longrightarrow\,
\bigwedge\nolimits^r {\mathcal O}^{\oplus r}_X\,=\, {\mathcal O}_X
\, \longrightarrow\, (\bigwedge\nolimits^r {\mathcal O}^{\oplus r}_X)/
(\bigwedge\nolimits^r {\mathcal K}_Q)\, :=\, Q'\, \longrightarrow\, 0
$$
be the short exact sequence obtained from it by taking $r$-th exterior
product. Sending any $Q\, \in\, {\mathcal Q}$ to $Q'\, \in\, {\mathcal Q}(1,d)$
constructed as above from $Q$ we get a morphism
\begin{equation}\label{e1}
f\, :\, {\mathcal Q}\,\longrightarrow\, S^d(X)\,=\,
{\mathcal Q}(1,d)\, .
\end{equation}

\begin{lemma}\label{lem1}
Take any point $z\,\in\, S^d(X)$. There is a sequence of holomorphic maps
ending at the point $z$,
$$
Y^z_d\,\stackrel{p_d}{\longrightarrow}\,
Y^z_{d-1} \,\stackrel{p_{d-1}}{\longrightarrow}\, \cdots
\,\stackrel{p_2}{\longrightarrow}\, Y^z_1
\,\stackrel{p_1}{\longrightarrow}\, Y^z_0 \,=\, \{z\}\, ,
$$
such that
\begin{itemize}
\item all the fibers of each $p_i$ are isomorphic to ${\mathbb C}{\mathbb P}^{r-
1}$, and

\item there is a surjective holomorphic map from $Y^z_d$ to the fiber
$f^{-1}(z)$.
\end{itemize}
\end{lemma}

\begin{proof}
Fix a point
$$
(x_1\, ,\cdots\, ,x_d)\,\in\, X^d
$$
that projects to $z$ under the quotient map $X^d\, \longrightarrow\,
X^d/P_d\,=\, S^d(X)$.
For any integer $i\, \in\, [1\, , d]$, let $Y^z_i$ be the space of all
filtrations of coherent analytic subsheaves
\begin{equation}\label{e-1}
W_i\, \subset\, W_{i-1} \, \subset\, \cdots \, \subset\,
W_{1} \, \subset\, W_0\,=\, {\mathcal O}^{\oplus r}_X\, ,
\end{equation}
where $W_j/W_{j+1}$, $0\,\leq\, j\, \leq\, i-1$, is a torsion sheaf
supported at $x_{j+1}$ such that the dimension of the stalk of
$W_j/W_{j+1}$ at $x_{j+1}$ is one.
Therefore, $W_{j+1}$ is a holomorphic vector bundle of rank $r$ and degree
$-j-1$. In particular, $W_i$ is a holomorphic vector bundle of rank $r$ and
degree $-i$. We have the natural forgetful map
$$
p_i \, :\, Y^z_i\, \longrightarrow\, Y^z_{i-1}
$$
that forgets the smallest subsheaf in the filtration (it forgets
$W_i$ in \eqref{e-1}). The inverse
image, under the projection $p_i$, of a point
$$
W_{i-1} \, \subset\, \cdots \, \subset\,
W_{1} \, \subset\, W_0\,=\, {\mathcal O}^{\oplus r}_X
$$
of $Y^z_{i-1}$ is the projective space ${\mathbb P}(W_{i-1})_{x_i}$ that
parametrizes all hyperplanes in the fiber of the vector bundle $W_{i-1}$ over
the point $x_i$. In particular, the map $p_i$ makes $Y^z_i$ a projective bundle
over $Y^z_{i-1}$ of relative dimension $r-1$.

Sending any point
$$
W_d\, \subset\, W_{d-1} \, \subset\, \cdots \, \subset\,
W_{1} \, \subset\, W_0\,=\, {\mathcal O}^{\oplus r}_X\, ,
$$
of $Y^z_d$ to the subsheaf $W_d\, \subset\, {\mathcal O}^{\oplus r}_X$
we get a holomorphic map
$$
Y^z_d\,\longrightarrow\, f^{-1}(z)\, .
$$
This map is clearly surjective.
\end{proof}

Let
\begin{equation}\label{e2}
\varphi\, :\, S^d(X)\,\longrightarrow\, \text{Pic}^d(X)
\end{equation}
be the morphism that sends a divisor to the corresponding holomorphic
line bundle on $X$.

\begin{corollary}\label{cor1}
The composition $\varphi\circ f\, :\, {\mathcal Q}\,\longrightarrow\,
{\rm Pic}^d(X)$ is the Albanese morphism for the variety ${\mathcal Q}$.
\end{corollary}

\begin{proof}
Since there is no nonconstant morphism from a rational curve to an
abelian variety, from Lemma \ref{lem1} it follows immediately that the
Albanese morphism for ${\mathcal Q}$ factors through $f$: Take any holomorphic 
line bundle $L\, \in\, {\rm Pic}^d(X)$.
Note that the fiber $\varphi^{-1}(L)$ is the
projective space ${\mathbb P}H^0(X,\, L)^*$. Therefore, the
Albanese morphism for ${\mathcal Q}$ factors through $\varphi\circ f$.
\end{proof}

\section{Holomorphic bisectional curvatures}

We continue with the notation of Section \ref{se2}.

\begin{theorem}\label{thm1}
The complex manifold $\mathcal Q$ does not admit any K\"ahler metric such that
all the holomorphic bisectional curvatures are nonnegative.
\end{theorem}

\begin{proof}
Assume that $\mathcal Q$ admits a K\"ahler metric $g_Q$ with nonnegative
holomorphic bisectional curvatures. Let $\omega$ denote the K\"ahler form of $g_Q$.

Let
\begin{equation}\label{rho}
\rho\, :\, \widetilde{\mathcal Q}\,\longrightarrow\, \mathcal Q
\end{equation}
be the universal cover. Then, by Mok's uniformization theorem for compact K\"ahler manifolds
with nonnegative bisectional curvature \cite[p. 179, Main Theorem]{mok}, the K\"ahler
manifold $(\widetilde{\mathcal Q} \, , \rho^\ast g_Q)$ decomposes as a Cartesian product
$$
(\widetilde{\mathcal Q}\, , \rho^\ast g) \,=\, ({\mathbb C}^N\, , g_0) \times 
(\prod_{i=1}^k({\mathbb C} P^{r_i}\, , g_i)) \times (\prod_{j=1}^\ell( (H_j\, ,h_j))
$$
of K\"ahler manifolds for some nonnegative integers $r_1\, ,\cdots \, ,r_k, , \ell$; here 
$({\mathbb C}^N\, , g_0)$ is the Euclidean space with its standard complex structure and 
metric, $g_i$ are K\"ahler metrics of nonnegative bisectional curvature on
${\mathbb C} P^{r_i}$, and $h_j$ are the standard symmetric K\"ahler metrics on certain 
Hermitian symmetric spaces $H_j$ of compact type. In what follows we do not need the 
specific information on the non-Euclidean factors in the decomposition.
We write 
$$
(\widetilde{\mathcal Q}\, , \rho^\ast g_Q) \,=\, ({\mathbb C}^N\, , g_0) \times (H\, , g')\, ,
$$
where $(H,g')\,=\, (\prod_{i=1}^k({\mathbb C} P^{r_i}\, , g_i)) \times (\prod_{j=1}^\ell( 
(H_j\, ,h_j))$. 

Let
$$
q_1\, :\, {\mathbb C}^N\times H\, \longrightarrow\, {\mathbb C}^N
$$
be the projection to the first factor. Let
$$
\gamma\, \in\, \text{Gal}(\rho)
$$
be a deck transformation for the covering $\rho$ in \eqref{rho}. Since there is no nonconstant
holomorphic map from $H$ to ${\mathbb C}^N$, there is a holomorphic isometry
$\gamma'\, \in\, \text{Aut}({\mathbb C}^N)$ such that
$$
q_1\circ \gamma\,=\, \gamma'\circ q_1\, .
$$
Therefore, $q_1$ descends to a holomorphic submersion of the form
\begin{equation}\label{beta}
\beta\, :\, {\mathcal Q}\, \longrightarrow\, Y
\end{equation}
such that $Y$ is (regularly) covered by ${\mathbb C}^N$; so
${\mathbb C}^N$ is a universal cover of $Y$. More precisely,
there is a K\"ahler form preserving covering map
$$
\rho'\, :\, {\mathbb C}^N\,\longrightarrow\, Y
$$
such that $\rho'\circ q_1\,=\, \beta\circ\rho$. We note that $\beta$ makes
$\mathcal Q$ a holomorphic fiber bundle over $Y$ with fibers isomorphic to $H$.

Consider the $C^\infty$ distribution
$$
{\mathcal S}\subset\, T\mathcal Q
$$
on $\mathcal Q$ obtained by taking the
orthogonal complement, with respect to $\omega$, of the relative tangent bundle
for the projection $\beta$ in \eqref{beta}. The pulled back distribution
$$
\rho^\star{\mathcal S}\,=\, (d\rho)^{-1}({\mathcal S})\,\subset\, T\widetilde{\mathcal Q}
$$
coincides with the relative tangent bundle for the projection $q_2\, :\,
{\mathbb C}^N\times H\,\longrightarrow\, H$ to the second factor; here $d\rho\, :\,
T\widetilde{\mathcal Q}\,\longrightarrow\, T{\mathcal Q}$ is the differential
of $\rho$. In particular,
$\rho^\star{\mathcal S}$ is an integrable holomorphic distribution. In
other words, $\rho^\star{\mathcal S}$ defines a flat holomorphic connection on the
fiber bundle
$$
q_1\, :\, {\mathbb C}^N\times H\, \longrightarrow\, {\mathbb C}^N\, .
$$
This implies
that ${\mathcal S}$ is a flat holomorphic connection on the fiber bundle
$\beta\, :\, {\mathcal Q}\, \longrightarrow\, Y$. Let
\begin{equation}\label{f1}
\nabla
\end{equation}
denote this flat holomorphic connection on $\beta\, :\, {\mathcal Q}\,
\longrightarrow\, Y$.

We now need the following proposition.

\begin{proposition}\label{prop1}
The fibers of the map $\beta$ coincides with the fibers of the
map $\varphi\circ f$. In particular, there is a biholomorphism
$$
\alpha\, :\, Y\, \longrightarrow\, {\rm image}(\varphi\circ f)
$$
such that $\alpha\circ \beta\,=\, \varphi\circ f$.
\end{proposition}

\begin{proof}[Proof of Proposition \ref{prop1}]
Take a point $y\, \in\, Y$. Consider the restriction of the map
$\varphi\circ f$ to $\beta^{-1}(y)$. Since $\beta^{-1}(y)\,=\, H$
is simply connected, this map
$$
(\varphi\circ f)\vert_{\beta^{-1}(y)}\, :\, \beta^{-1}(y)\, \longrightarrow
\, \text{Pic}^d(X)
$$
lifts to a map
$$
\beta^{-1}(y)\, \longrightarrow \, \widetilde{\text{Pic}}^d(X)\, ,
$$
where $\widetilde{\text{Pic}}^d(X)$ is the universal cover of $\text{Pic}^d(X)$.
But $\widetilde{\text{Pic}}^d(X)$ is isomorphic to ${\mathbb C}^{^g}$,
and hence any holomorphic map to it from $\beta^{-1}(y)$ is a constant map.

Conversely, take any point $L\, \in\, \varphi\circ f(\mathcal Q)$. We will
show that the restriction of the map $\beta$ to the fiber $(\varphi\circ
f)^{-1}(L)$ is a constant map. For this we first show that there is no nonconstant
holomorphic map from ${\mathbb C}{\mathbb P}^1$ to $Y$. To prove this note that
${\mathbb C}{\mathbb P}^1$ is simply connected and the universal over of $Y$ is
${\mathbb C}^{N}$. Therefore, any holomorphic map from ${\mathbb C}{\mathbb P}^1$ to $Y$
lifts to a holomorphic map from ${\mathbb C}{\mathbb P}^1$ to ${\mathbb C}^{N}$, but there
are no such nonconstant holomorphic maps.

Since there is no nonconstant holomorphic map from ${\mathbb C}{\mathbb P}^1$
to $Y$, from Lemma \ref{lem1} it follows that the map
$$
\beta\vert_{(\varphi\circ f)^{-1}(L)}\, :\,
(\varphi\circ f)^{-1}(L)\,\longrightarrow\, Y
$$
factors through $ (\varphi\circ f)^{-1}(L)\, \stackrel{f}{\longrightarrow}
\, \varphi^{-1}(L)$. As noted in the proof of Corollary \ref{cor1}, the fiber
$\varphi^{-1}(L)$ is a projective space. Hence
the restriction of $\beta$ to the fiber $(\varphi\circ f)^{-1}(L)$
is a constant map. This completes the proof of the proposition.
\end{proof}

Continuing with the proof of the theorem, since $g\, \geq\, 2$, we know 
that $\mathcal Q (r,d)$ does not admit any K\"ahler metric with nonnegative holomorphic
bisectional curvature if $d \, <\, g$ \cite[Corollary 4.1]{BS}.
Therefore, we assume that $d\, \geq\, g$. Consequently, the map $\varphi$ is surjective.
Since $f$ is surjective, this implies that $\varphi\circ f$ is surjective.
Now from Proposition \ref{prop1} we know that $Y$ is an abelian variety.

Choose a point of $Y$ to make it a group. The universal covering map
${\mathbb C}^N\,\longrightarrow\, Y$ is chosen to be a homomorphism of groups.
For any point $v$ in the
universal cover ${\mathbb C}^N$ of $Y$, let
$$
\tau_v\, :\, Y\, \longrightarrow\, Y
$$ be the automorphism given by the automorphism $w\, \longmapsto\, w+v$ of ${\mathbb C}^N$.
We have a one-parameter family of automorphisms $Y$ given by $\tau_{tv}$, $t\,
\in\, [0\, ,1]$, that connects $\tau_v$ with the identity map of $Y$.

Consider the flat holomorphic connection $\nabla$ in \eqref{f1}.
Given a point $y\, \in\, Y$, we may take parallel translation of the fiber of $\beta$
over $y$ along the path $\tau_{tv}(y)$, $t\, \in\, [0\, ,1]$.
Taking parallel translations for $\nabla$ along such paths
$\tau_{tv}$, $t\, \in\, [0\, ,1]$, we get holomorphic automorphisms of
$\mathcal Q$. If $v$ is sufficiently close to zero, then the automorphism
$\tau_v$ of $Y$ clearly does not have any fixed-point. Hence for $v$ sufficiently
close to zero, we get
holomorphic automorphisms of $\mathcal Q$ that do not have any fixed point.

Let $\text{Aut}(\mathcal Q)$ denote the group of holomorphic automorphisms
of $\mathcal Q$. Let $\text{Aut}^0(\mathcal Q)\, \subset\,\text{Aut}(\mathcal
Q)$ be the connected component of it containing the identity element.

The standard action of $\text{GL}(r,{\mathbb C})$ on ${\mathbb C}^r$ produces
an action of $\text{GL}(r,{\mathbb C})$ on the trivial vector bundle
${\mathcal O}^{\oplus r}_X\,=\, X\times {\mathbb C}^r$. The corresponding
action of $\text{GL}(r,{\mathbb C})$ on $\mathcal Q$ factors through the quotient
group $\text{PGL}(r,{\mathbb C})\,=\, \text{GL}(r,{\mathbb C})/{\mathbb C}^*$. Since
this action of $\text{PGL}(r, {\mathbb C})$ on $\mathcal Q$ is effective, we have
$$
\text{PGL}(r,{\mathbb C})\, \subset\, \text{Aut}^0(\mathcal Q)\, .
$$
It is known that $\text{PGL}(r,{\mathbb C})\, =\, \text{Aut}^0(\mathcal Q)$
\cite[Theorem 1.1]{BDH}. Since the standard action on ${\mathbb C}{\mathbb
P}^{r-1}$ of any element $A\, \in\, \text{PGL}(r,{\mathbb C})$ has a fixed
point, the action of $A$ on $\mathcal Q$ has a fixed point.

This contradicts our earlier construction of automorphisms of
$\mathcal Q$ without fixed points. In view of this contradiction, we conclude
that $\mathcal Q$ does not admit any K\"ahler structure such that all the
holomorphic bisectional curvatures are nonnegative.
\end{proof}



\begin{thebibliography}{ZZZZ}

\bibitem[Ba]{Ba} J. M. Baptista, On the $L^2$--metric of vortex moduli spaces,
{\it Nuclear Phys. B} \textbf{844} (2011), 308--333.

\bibitem[BGL]{BGL} E. Bifet, F. Ghione, and M. Letizia, On the Abel-Jacobi map
for divisors of higher rank on a curve, {\it Math. Ann.} \textbf{299} (1994), 641--672.

\bibitem[BDH]{BDH} I. Biswas, A. Dhillon and J. Hurtubise, Automorphisms
of the Quot schemes associated to compact Riemann surfaces,
\textit{Int. Math. Res. Not.}, doi: 10.1093/imrn/rnt259.

\bibitem[BiR]{BiR} I. Biswas and N. M. Rom\~{a}o, Moduli of vortices and Grassmann
manifolds, {\it Comm. Math. Phys.} \textbf{320} (2013), 1--20.

\bibitem[BS]{BS} I. Biswas and H. Seshadri, On the K\"ahler structures over Quot
schemes, \textit{Ill. Jour. Math.} \textbf{57} (2013), 1019--1024.

\bibitem[BoR]{BR} M. B\"okstedt and N. M. Rom\~{a}o, On the curvature
of vortex moduli spaces, \textit{Math. Zeit.} {\bf 277} (2014), 549--573.

\bibitem[Mo]{mok} N. Mok, The uniformization theorem for compact K\"ahler
manifolds of nonnegative holomorphic bisectional curvature, \textit{Jour.
Diff. Geom.} {\bf 27} (1988), 179--214.

\end{thebibliography}
\end{document}